\documentclass[12pt]{amsart}

\usepackage{fullpage}
\usepackage{amsmath}
\usepackage{amssymb}
\usepackage{amsthm}
\usepackage{amstext}
\usepackage{appendix}
\usepackage{perpage}
\usepackage[x11names]{xcolor}
\usepackage{tikz-cd}
\usepackage{ytableau}\ytableausetup{centertableaux}
\usepackage[colorlinks=true, allcolors=blue]{hyperref}
\usepackage{todonotes}

\newtheorem{thm}{Theorem}[section]
\newtheorem{lemm}[thm]{Lemma}
\newtheorem{coro}[thm]{Corollary}
\newtheorem{prop}[thm]{Proposition}
\newtheorem{conj}[thm]{Conjecture}

\theoremstyle{definition}

\newtheorem{remark}[thm]{Remark}

\begin{document}

\title{On a stability of higher level Coxeter unipotent representations}

\author{Zhe Chen}

\address{Department of Mathematics, Shantou University, Shantou, 515821, China}
\email{zhechencz@gmail.com}

\begin{abstract}
Let  $\mathbb{G}$ be a connected reductive group over $\mathcal{O}$, a complete discrete valuation ring with finite residue field $\mathbb{F}_q$. Let $R_{T_r,U_r}^{\theta}$ be a level $r$ Deligne--Lusztig representation of $\mathbb{G}(\mathcal{O})$, where $r$ is a positive integer. We show that, if $q$ is not small, and if $T$ is Coxeter and $\theta=1$, then $R_{T_r,U_r}^1$ degenerates to the $r=1$ case. For $\mathbb{G}=\mathrm{GL}_2$ (or $\mathrm{SL}_2$), as an application we give the dimensions and decompositions of all  $R_{T_r,U_r}^{\theta}$ for Coxeter $T$. This in turn leads us to state a conjectural sign formula for $R_{T_r,U_r}^{\theta}$, for general $(\mathbb{G}, T, \theta,r)$.
\end{abstract}

\maketitle

\section{Introduction and preliminaries}

Let $\mathcal{O}$ be a complete discrete valuation ring with residue field $\mathbb{F}_q$, and let $\mathbb{G}$ be a connected reductive group over $\mathcal{O}$. Fix a uniformiser $\pi$ of $\mathcal{O}$. Then every smooth irreducible representation of $\mathbb{G}(\mathcal{O})$ factors through the finite quotient group $\mathbb{G}(\mathcal{O}_r)$ for some $r\in\mathbb{Z}_{>0}$, where $\mathcal{O}_r:=\mathcal{O}/\pi^r$.

\vspace{2mm} When $r=1$, a family of virtual representations was constructed in \cite{DL1976} by $\ell$-adic cohomology, which realises all irreducible representations of $\mathbb{G}(\mathbb{F}_q)$. In 1979,  Lusztig \cite{Lusztig1979SomeRemarks} generalises this construction to every $r$, with the proofs been completed later in \cite{Lusztig2004RepsFinRings} (for $\mathrm{char}(\mathcal{O})>0$) and \cite{Sta2009Unramified} (in general). However, when $r>1$ it is known that not all irreducible representations appear in this construction, so two fundamental questions arise:
\begin{itemize}
\item How to generalise \cite{Lusztig1979SomeRemarks} to recover the missing representations?
\item How to characterise/understand the representations given in \cite{Lusztig1979SomeRemarks}?
\end{itemize}
The first question has been investigated in \cite{Sta2011ExtendedDL} and \cite{Chen_2019_flag_orbit}, using  tamely ramified torus and combinatorial analogues, respectively, for $\mathrm{GL}_n$ and $\mathrm{SL}_n$; in some special cases it is shown that the generalised constructions therein afford all irreducible representations. 

\vspace{2mm} The second question is answered for generic parameters, in  \cite{ChenStasinski_2016_algebraisation} and \cite{ChenStasinski_2023_algebraisation_II}, by establishing an explicit algebraic realisation  for the geometrically-constructed virtual representations given in \cite{Lusztig1979SomeRemarks}.  Here, the notion of generic parameters is an analogue of the notion of regular semisimple elements (in a Lie algebra), and for $\mathrm{GL}_n$ this analogue can be phrased to be a precise relation via an orbit map; in particular, almost all parameters are generic. For details, see \cite[Proposition~3.9, Theorem~3.7, and Proposition~2.2]{ChenStasinski_2023_algebraisation_II}.

\vspace{2mm} However, the non-generic parameter case is not well-understood, and the original motivation of this  note is to answer the second question with the trivial parameter in the Coxeter setting. Indeed, when $r=1$,  the involved irreducible representations  have been studied in  \cite{Lusztig_1976_CoxOrbitsFrob}, in which it is shown that they are parametrised by the eigenvalues of the Frobenius action. Since then, these representations play an important role in understanding the representation theory of finite groups of Lie type. Here we call them the Coxeter unipotent representations, in which the term ``unipotent'' comes from Lusztig's Jordan decomposition of characters (see \cite{Lusztig_84_OrangeBook}). For $r>1$, the corresponding representations form a natural candidate of higher level Coxeter unipotent representations, and shall be expected to play a similar role. However, our first result is a negative one, showing that this candidate degenerates to the $r=1$ case. 

\vspace{2mm} As such, the notion of Coxeter unipotent representations for $r>1$ must be revised, and it would be interesting to see whether the constructions in \cite{Sta2011ExtendedDL} or \cite{Chen_2019_flag_orbit} can lead to a correct notion of higher level Coxeter unipotent representations. In this short note we do not investigate this aspect, instead, we will use the above results to make some computations in the special case of $\mathrm{GL}_2$ and $\mathrm{SL}_2$, which in turn suggests a neat conjectural formula linking the signs and dimensions of Lusztig's virtual representations (for any parameter).

\vspace{2mm} To state precisely our works in this note, let us give a  brief recall of the construction in \cite{Lusztig1979SomeRemarks}: Let $\mathcal{O}^{\mathrm{ur}}$ be the ring of integers in a maximal unramified extension of $\mathrm{Frac}(\mathcal{O})$, and let $U\rtimes T$ be a Levi decomposition of a Borel subgroup of $\mathbb{G}_{\mathcal{O}^{\mathrm{ur}}}$, where $U$ is the unipotent radical and $T$ a maximal torus. Then  there is a natural algebraic group structure on $\mathbb{G}(\mathcal{O}^{\mathrm{ur}}/\pi^r)$ (resp.\  $U(\mathcal{O}^{\mathrm{ur}}/\pi^r)$, $T(\mathcal{O}^{\mathrm{ur}}/\pi^r)$), denoted by $G_r$ (resp.\ $U_r$, $T_r$); there is a geometric Frobenius endomorphism $F$ on $G_r$  satisfying that $G_r^F\cong \mathbb{G}(\mathcal{O}_r)$ as finite groups. Let $L$ be the Lang map associated to $F$. Assume that $T_r=FT_r$. Then there is a virtual $G_r^F$-representation
$$R_{T_r,U_r}^{\theta}:=\sum_{i}(-1)^iH_c^i(L^{-1}(FU_r),\overline{\mathbb{Q}}_{\ell})_{\theta}$$
for each $\theta\in\mathrm{Irr}(T_r^F)$, where $H^i_c(-,\overline{\mathbb{Q}}_{\ell})_{\theta}$ denotes the $\theta$-isotypical part of compactly-supported $\ell$-adic cohomology group (here $\ell\nmid q$). For details, see \cite{Lusztig2004RepsFinRings} and \cite{Sta2009Unramified}.

\vspace{2mm} Suppose that $\theta=1$ and $T$ is Coxeter. When $r=1$ the representations inside $R_{T_1,U_1}^1$ are studied by Lusztig in the seminal work \cite{Lusztig_1976_CoxOrbitsFrob}. For $r\geq2$ it is natural to expect that $R_{T_r,U_r}^1$ will provide further interesting  representations. However, Lusztig's computation in \cite[Subsection~3.4]{Lusztig2004RepsFinRings} indicates that this is not the case for  $\mathrm{SL}_2(\mathbb{F}_q[[\pi]]/\pi^2)$. In  Theorem~\ref{thm:main} (and Proposition~\ref{prop:main}) we give an extension of this phenomenon for general $\mathbb{G}$ and $r\geq 2$. A new insight of our argument is an inner product linking different levels (Proposition~\ref{lemma:inner product linking different levels}).

\vspace{2mm} For general $\theta$, it is desirable to get the dimensions and decompositions of $R_{T_r,U_r}^{\theta}$; actually, the $\mathrm{SL}_2(\mathbb{F}_q[[\pi]]/\pi^2)$ case treated in \cite[Subsection~3.4]{Lusztig2004RepsFinRings} has already been very useful on testing potential properties of $R_{T_r,U_r}^{\theta}$. In Proposition~\ref{prop:GL_2}, based on the above result and earlier works, we produce these information of Coxeter $R_{T_r,U_r}^{\theta}$ for $\mathrm{GL}_2(\mathcal{O}_r)$, for any  $r\geq2$. Then in Section~\ref{sec:remark} we present a conjectural sign formula for  $R_{T_r,U_r}^{\theta}$ for a general quadruple $(\mathbb{G},T,\theta,r)$.

\vspace{2mm} \noindent {\bf Acknowledgement.} The author thanks Alexander Stasinski and  Sian Nie for helpful conversations around earlier versions of this short note. During the preparation of this work, ZC is partially supported by the Natural Science Foundation of Guangdong No.~2023A1515010561 and NSFC No.~12171297.

\section{Stabilities of inner products and representations}

Given a finite group $H$ and a normal subgroup $N$, if $\chi$ is a representation of $H/N$, we shall denote by $\widetilde{\chi}$ the composition of $\chi$ with the quotient map $H\rightarrow H/N$. So $\widetilde{\chi}$ is a representation of $H$, and it is usually called the trivial inflation of $\chi$ along $H\rightarrow H/N$.

\vspace{2mm} For $r'$ a positive integer not greater than $r$, we shall write $\rho_{r'}\colon G_r\rightarrow G_{r'}$ for the reduction map modulo $\pi^{r'}$.

\begin{thm}\label{thm:main}
Assume that $q\geq7$ or $\mathbb{G}=\mathrm{GL}_n$, and assume that $T$ is Coxeter. Then for every $r\geq1$:
$$R_{T_r,U_r}^1\cong  \widetilde{R_{T_1,U_1}^1}$$
as representations of $G_r^F$, where the right hand side is the trivial inflation along $\rho_1\colon G_r\rightarrow G_1$.
\end{thm}

The proof is a combination of inner product formulae, of which the starting point is:

\begin{prop}[Deligne--Lusztig]\label{lemma:DL}
We have
\begin{equation*}
\langle   R_{T_1,U_1}^{\theta}, R_{T_1,U_1}^{\theta}    \rangle_{G_1^F}=\# \{w\in (N_{G_1}(T_1)/T_1)^F\mid {^w\theta}=\theta\}.
\end{equation*}
\end{prop}
\begin{proof}
This is a special case of \cite[Theorem~6.8]{DL1976}.
\end{proof}

For us, we need the following extension given by Chan--Ivanov \cite{ChanIvanov_CambJ_2023_loopDL_GL_n} and Dudas--Ivanov \cite{Dudas_Ivanov_Orthogonality_Coxeter}:

\begin{prop}[Chan--Ivanov, Dudas--Ivanov]\label{lemma:DI}
Assume that $q\geq7$ or $\mathbb{G}=\mathrm{GL}_n$, and assume that $T$ is Coxeter. Then
\begin{equation*}
\langle  {R_{T_r,U_r}^{\theta}}, R_{T_r,U_r}^{\theta}     \rangle_{G_r^F}=\# \{w\in (N_{G_r}(T_r)/T_r)^F\mid {^w\theta}=\theta\}.
\end{equation*}
\end{prop}
\begin{proof}
For $\mathrm{GL}_n$ this is a special case of \cite[Theorem~3.1]{ChanIvanov_CambJ_2023_loopDL_GL_n}, and in general this is a special case of \cite[Theorem~3.2.3]{Dudas_Ivanov_Orthogonality_Coxeter}.
\end{proof}

We also need an inner product linking Deligne--Lusztig representations at different levels: (Note that here $T$ is not assumed to be Coxeter.)

\begin{prop}\label{lemma:inner product linking different levels}
Let $r'$ be a positive integer not greater than $r$. For  $\theta\in\mathrm{Irr}(T_{r}^F)$ and $\theta'\in\mathrm{Irr}(T_{r'}^F)$ we have:
\begin{itemize}
\item[(a)] If $\theta$  is the trivial inflation of a character $\theta''\in\mathrm{Irr}(T_{r'}^F)$, then
\begin{equation*}
\langle   \widetilde{R_{T_{r'},U_{r'}}^{\theta'}}, R_{T_r,U_r}^{\theta}     \rangle_{G_r^F}
=\langle   R_{T_{r'},U_{r'}}^{\theta'}, R_{T_{r'},U_{r'}}^{\theta''}   \rangle_{G_{r'}^F};
\end{equation*}
\item[(b)] if $\theta$ is non-trivial on $\mathrm{Ker}(\rho_{r'})\cap T_r^F$, then
\begin{equation*}
\langle   \widetilde{R_{T_{r'},U_{r'}}^{\theta'}}, R_{T_r,U_r}^{\theta}     \rangle_{G_r^F}=0.
\end{equation*}
\end{itemize}
\end{prop}

\begin{proof}
Let $L'$ be the Lang map on $G_{r'}$ (we shall still use $F$ to denote the geometric Frobenius on $G_{r'}$). By \cite[7.1(b)]{serre1977lin_rep_finite_gr} we have 
$$\langle \widetilde{R_{T_{r'},U_{r'}}^{\theta'}}, R_{T_r,U_r}^{\theta} \rangle_{G_r^F}
=\langle  R_{T_{r'},U_{r'}}^{\theta'}, {^{\mathrm{Ker(\rho_{r'})}^F}R_{T_r,U_r}^{\theta}}    \rangle_{G_{r'}^F},$$
where the symbol ${^{\mathrm{Ker(\rho_{r'})}^F}(-)}$ means taking the subspace of $\mathrm{Ker(\rho_{r'})}^F$-invariant vectors. By the K\"{u}nneth formula the RHS in the above is equal to
\begin{equation}\label{formula:temp1}
\sum_i(-1)^i \dim 
H_c^i\left(
G_{r'}^F   \backslash 
\left(L'^{-1}(FU_{r'})\times   (\mathrm{Ker(\rho_{r'})}^F \backslash L^{-1}(FU_r))\right) 
\right)_{\theta'^{-1}\times \theta}.
\end{equation}
Now consider
$$\Sigma:=\{ (u,v,y)\in FU_{r'}\times FU_r\times G_{r'}\mid uF(y)=y\rho_{r'}(v)  \},$$
on which $T_{r'}^F\times T_r^F$ acts (from the right hand side) by
$$(t',t)\colon (u,v,y)\longmapsto (u^{t'},{v}^{t},{t'}^{-1}y\rho_{r'}(t)).$$
Note that the morphism
$${L'}^{-1}(FU_{r'})\times   L^{-1}(FU_r)\longrightarrow\Sigma,$$
given by
$$(g',g)\longmapsto (L'(g'),L(g),{g'}^{-1}\rho_{r'} (g)),$$
is surjective by the Lang--Steinberg theorem; then a direct computation shows that it induces a $T_{r'}^F\times T_r^F$-equivariant bijection
\begin{equation*}
G_{r'}^F\backslash \left({L'}^{-1}(FU_{r'})\times  (\mathrm{Ker}(\rho_{r'})^F\backslash L^{-1}(FU_r))\right)\longrightarrow \Sigma.
\end{equation*}
So (see e.g.\ \cite[8.1.13]{DM_book_2nd_edition})
$$\eqref{formula:temp1}=\sum_i(-1)^i\dim H_c^i(\Sigma)_{\theta'^{-1}\times\theta}.$$
Meanwhile, by the reduction map there is a $T_{r'}^F\times T_r^F$-equivariant surjection from $\Sigma$ to
$$\Sigma':=\{ (u,v,y)\in FU_{r'}\times FU_{r'}\times G_{r'}\mid uF(y)=yv  \},$$
on which $T_{r'}^F\times T_r^F$ acts through the quotient $T_{r'}^F\times T_{r'}^F$, whose fibres are isomorphic to an affine space ($\cong \mathrm{Ker}(\rho_{r'}|_{FU_r})$). So (\cite[8.1.13]{DM_book_2nd_edition}) we have
$$\eqref{formula:temp1}=\sum_i(-1)^i\dim H_c^i(\Sigma')_{\theta'^{-1}\times\theta}.$$

\vspace{2mm} For the cohomology of $\Sigma'$, again by  the K\"unneth formula we have (see also the argument of \cite[Proposition~2.2]{Lusztig2004RepsFinRings})
\begin{equation*}
\begin{split}
 \sum_{i}(-1)^i \dim & H^i_c(\Sigma')_{\theta'^{-1}\times\theta}\\
&= \left\langle
\sum_{i}(-1)^iH_c^i(L^{-1}(FU_{r'}),\overline{\mathbb{Q}}_{\ell})_{\theta'},\sum_{i}(-1)^iH_c^i(L^{-1}(FU_{r'}),\overline{\mathbb{Q}}_{\ell})_{\theta}
\right\rangle_{G_{r'}^F},
\end{split}
\end{equation*}
in which the right action of $T_r^F$ on $H_c^i(L^{-1}(FU_{r'}),\overline{\mathbb{Q}}_{\ell})$ is with respect to the quotient $T_r^F\rightarrow T_{r'}^F$. Now the assertion follows immediately.
\end{proof}

\begin{proof}[Proof of Theorem~\ref{thm:main}]
Let $W$ be the Weyl group generated by the simple reflections with respect to the root system of $\mathbb{G}_{\mathcal{O}^{\mathrm{ur}}}$ relative to $T$. Then according to \cite[XXII~3.4]{SGA3} there are natural isomorphisms $W\cong N_{G_r}(T_r)/T_r$ and $W\cong N_{G_1}(T_1)/T_1$, which are, by the construction in \cite[XXII~3.3]{SGA3}, compatible with the reduction map modulo $\pi$; since the reduction map commutes with $F$, this implies that $(N_{G_r}(T_r)/T_r)^F\cong (N_{G_1}(T_1)/T_1)^F$. Now taking $r'=1$ and $\theta=\widetilde{\theta'}=1$, in Proposition~\ref{lemma:DI} and Proposition~\ref{lemma:inner product linking different levels}, we get
$$|(N_{G_r}(T_r)/T_r)^F|
=\langle   R_{T_r,U_r}^1 , R_{T_r,U_r}^1      \rangle_{G_r^F}
=\langle   \widetilde{R_{T_1,U_1}^1}, R_{T_r,U_r}^1      \rangle_{G_r^F}
=\langle   \widetilde{R_{T_1,U_1}^1} , \widetilde{R_{T_{1},U_{1}}^1}    \rangle_{G_r^F},$$
which implies that
$$\langle   \widetilde{R_{T_1,U_1}^1}-R_{T_r,U_r}^1 , \widetilde{R_{T_{1},U_{1}}^1}-R_{T_r,U_r}^1      \rangle_{G_r^F}=0,$$
so $R_{T_r,U_r}^1\cong \widetilde{R_{T_1,U_1}^1}$.
\end{proof}

A slightly more general result follows from exactly the same method in the above:

\begin{prop}\label{prop:main}
Assume that $q\geq7$ or $\mathbb{G}=\mathrm{GL}_n$, and assume that $T$ is Coxeter. Given $r'\leq r$, if $\theta\in\mathrm{Irr}(T_{r}^F)$ is the trivial inflation of $\theta'\in\mathrm{Irr}(T_{r'}^F)$ along the reduction map $T_r^F\rightarrow T_{r'}^F$, then $R_{T_r,U_r}^{\theta}\cong  \widetilde{R_{T_{r'},U_{r'}}^{\theta'}}$.
\end{prop}

\begin{remark}\label{remark:Chan}
Note that in a recent new preprint, Charlotte Chan obtains this result (by a completely different method) in the general parahoric and elliptic setting, without assuming ``$q\geq7$''; see \cite[Theorem~5.2 and Corollary~5.3]{Chan_2024_Scalar_Product}. This in turn enables her to prove a generalisation of the inner product formulae in \cite{ChanIvanov_CambJ_2023_loopDL_GL_n},\cite{Dudas_Ivanov_Orthogonality_Coxeter}; in particular, after Chan's work the condition ``$q\geq7$'' in Proposition~\ref{lemma:DI} is no more required. For more details, see the argument of \cite[Proposition~6.4]{Chan_2024_Scalar_Product}.
\end{remark}

\section{The case of $\mathrm{GL}_2$}\label{sec:GL_2}

In this section we assume that $\mathbb{G}=\mathrm{GL}_n$ and $r\geq 2$. Write $\bar{\mathbb{G}}$ for the closed subgroup scheme $\mathrm{SL}_n$ of $\mathbb{G}$; similar notation $\bar{G}$ (resp.\ $\bar{T}, \bar{U}$, ...) applies to the corresponding closed subgroup of $G$ (resp.\ $T, U$, ...). Let $\bar{\theta}=\theta|_{\bar{T}_r^F}$. 

\begin{lemm}\label{lemm:GL_n to SL_n}
We have 
$$R_{\bar{T}_r,\bar{U}_r}^{\bar{\theta}}
=\mathrm{Res}^{G_r^F}_{\bar{G}_r^F}R_{T_r,U_r}^{\theta}.$$
\end{lemm}
\begin{proof}
This is proved in the argument of \cite[Proposition~4.1]{Chen_2019_flag_orbit}. (While  \cite{Chen_2019_flag_orbit} assumes $\mathrm{char}(\mathcal{O})>0$, the argument of this property works for any $\mathcal{O}$.)
\end{proof}

Let $\psi$ be an irreducible character of the additive group of $\mathbb{F}_q$. Write $T_r^{r-1}$ for the kernel of the reduction map $T_r\rightarrow T_{r-1}$; this can be viewed as the additive group of the Lie algebra $\mathfrak{t}$ of $T_1$. Then there is a unique $\tau_{\theta}\in \mathfrak{t}^F$ such that 
$$\theta(t)=\psi\left(\mathrm{Tr}\left(\frac{1}{\pi^{r-1}}(t-I)\cdot\tau_{\theta}\right)\right)$$ 
for any $t\in (T_r^{r-1})^F$, where $I$ denotes the identity matrix. We call $\theta$ \emph{regular} if $\tau_{\theta}$ is regular. For $\mathrm{GL}_n$ this regularity is equivalent to the regularity in \cite{Lusztig2004RepsFinRings} and \cite{Sta2009Unramified}, as can be seen from the argument of \cite[Proposition~2.2]{ChenStasinski_2023_algebraisation_II}.

\vspace{2mm} When $\theta$ is regular, the character of $R_{T_r,U_r}^{\theta}$ is explicitly determined in \cite[Theorem~4.4]{ChenStasinski_2023_algebraisation_II} (see also \cite[Proposition~2.2 and Remark~4.6]{ChenStasinski_2023_algebraisation_II}). So, according to Proposition~\ref{prop:main}, for $\mathrm{GL}_2$ one should focus on the case that $\theta$ is neither regular nor a trivial inflation; for this we will use Lemma~\ref{lemm:GL_n to SL_n} and the below lemma.

\begin{lemm}\label{lemm:conductor}
Assume that $n=2$. If $\theta$ is not regular, there is a positive integer $r'<r$ such that, for some   $\alpha\in\mathrm{Irr}(\mathcal{O}_r^{\times})$, the character $\theta\cdot\alpha(\det(-))$ is  the trivial inflation of some $\theta'\in\mathrm{Irr}(T_{r'}^F)$.
\end{lemm}
\begin{proof}
If $\theta$ is not regular, then $\tau_{\theta}$ is a central element in $\mathfrak{t}$, so $\tau_{\theta}=\mathrm{diag}(s,s)$ for some $s\in {\mathbb{F}}_q$, hence for any $t\in ({T}_r^{r-1})^F$ we have $\theta(t)=\psi(s\cdot\mathrm{Tr}(\frac{1}{\pi^{r-1}}(t-I)))$. However, for $t\in ({T}_r^{r-1})^F$ we have $\det(t)=1+\mathrm{Tr}(t-I)$, so
$$\theta(t)=\psi(s\cdot\frac{1}{\pi^{r-1}}(\det(t)-1)).$$
Note that $x\mapsto \frac{1}{\pi^{r-1}}(x-1)$ is a group isomorphism between the reduction kernel $K:=\mathrm{Ker}(\mathcal{O}_r^{\times}\rightarrow \mathcal{O}_{r-1}^{\times})$ to the additive group of $\mathbb{F}_q$, thus $\psi(s\cdot \frac{1}{\pi^{r-1}}((-)-1))$ is a character $\alpha'$ of $K$; since $\mathcal{O}_r^{\times}$ is abelian, $\alpha'$ extends to a character $\alpha''$ of $\mathcal{O}_r^{\times}$. Therefore, on $(T_r^{r-1})^F$ we have $\theta=\alpha''(\det)$, so $\theta\cdot{\alpha''}^{-1}(\det)$ is the trivial inflation of some $\theta'\in\mathrm{Irr}(T_{r'}^F)$ for some $r'<r$.
\end{proof}

\vspace{2mm} For $\mathbb{G}=\mathrm{GL}_2$, if $\theta$ is not regular, write $r_0$ for the smallest possible $r'$ in Lemma~\ref{lemm:conductor}, and write $\theta_0$ for the corresponding $\theta'$. Then Lemma~\ref{lemm:conductor} implies that either $r_0=1$, or $r_0>1$ and $\theta_0$ is regular (for $T_{r_0}^F$).

\begin{prop}\label{prop:GL_2}
Assume that $\mathbb{G}=\mathrm{GL}_2$ and $T$ is Coxeter. We have:
\begin{itemize}
\item[(i)] If $\theta$ is regular, then $(-1)^rR_{T_r,U_r}^{\theta}$ is irreducible and of dimension $(q-1)q^{r-1}$;
\item[(ii)] if $\theta$ is not regular and $r_0>1$, then $(-1)^{r_0}R_{T_r,U_r}^{\theta}$  is irreducible and of dimension $(q-1)q^{r_0-1}$; 
\item[(iii)] if $\theta$ is not regular, $r_0=1$, and $\theta_0$ is in general position, then $-R_{T_r,U_r}^{\theta}$ is irreducible and of dimension $q-1$; 
\item[(iv)] if $\theta$ is not regular, $r_0=1$, and $\theta_0$ is not in general position, then $R_{T_r,U_r}^{\theta}=\sigma_1-\sigma_2$, where $\sigma_1$ is a $1$-dimensional representation and $\sigma_2$ is a $q$-dimension irreducible representation.
\end{itemize}
\end{prop}
\begin{proof}
(i): This is just a special case of \cite[Theorem~4.4]{ChenStasinski_2023_algebraisation_II} (see also \cite[Proposition~2.2 and Remark~4.6]{ChenStasinski_2023_algebraisation_II}).

\vspace{2mm} (ii) and (iii): By Proposition~\ref{lemma:DI}, the virtual representation $R_{T_r,U_r}^{\theta}$ is irreducible up to a sign. Let $\widetilde{\theta_0}$ be the trivial inflation of $\theta_0$ to $T_r^F$. Then
$$\dim R_{T_r,U_r}^{\theta}=\dim R_{T_r,U_r}^{\widetilde{\theta_0}}=\dim R_{T_{r_0},U_{r_0}}^{{\theta_0}}=(-1)^{r_0}(q-1)q^{r_0-1},$$
where the first equality follows from Lemma~\ref{lemm:GL_n to SL_n} and Lemma~\ref{lemm:conductor}, the second equality follows from Proposition~\ref{prop:main}, and the third equality follows from the regular case (in (i)) and the well-known finite field case (see e.g.\ \cite[Section~12.5]{DM_book_2nd_edition}).

\vspace{2mm} (iv): By Proposition~\ref{lemma:DI}, the virtual representation $R_{T_r,U_r}^{\theta}$ has two inequivalent irreducible constituents  up to signs, and the same argument as above shows that  $\dim R_{T_r,U_r}^{\theta}=1-q$. Moreover, since in this case $\theta_0|_{\bar{T}_1^F}=1$, by Lemma~\ref{lemm:GL_n to SL_n} we get 
$$\mathrm{Res}^{G_r^F}_{\bar{G}_r^F}R_{T_r,U_r}^{\theta}=\mathrm{Res}^{G_r^F}_{\bar{G}_r^F}R_{T_r,U_r}^{\widetilde{\theta_0}}=\mathrm{Res}^{G_r^F}_{\bar{G}_r^F}R_{T_r,U_r}^{1}=1-\widetilde{\mathrm{St}},$$ 
where $\mathrm{St}$ stands for the Steinberg representation of $\bar{G}_1^F$; in particular, at least one irreducible constituent of $R_{T_r,U_r}^{\theta}$ has positive coefficient. So, since $\dim R_{T_r,U_r}^{\theta}<0$, we can write $R_{T_r,U_r}^{\theta}=\sigma_1-\sigma_2$ where the $\sigma_i$'s are inequivalent irreducible representations. As $\mathrm{Res}^{G_r^F}_{\bar{G}_r^F}\sigma_1$ contains the trivial representation and $\mathrm{Res}^{G_r^F}_{\bar{G}_r^F}\sigma_2$ contains $\widetilde{\mathrm{St}}$, by Clifford theory all irreducible constituents of $\mathrm{Res}^{G_r^F}_{\bar{G}_r^F}\sigma_1$ (resp.\ $\mathrm{Res}^{G_r^F}_{\bar{G}_r^F}\sigma_2$) are of dimension $1$ (resp.\ $q$). In particular, $\mathrm{Res}^{G_r^F}_{\bar{G}_r^F}\sigma_1$ and $\mathrm{Res}^{G_r^F}_{\bar{G}_r^F}\sigma_2$ have no common irreducible constituents. So the equality $\mathrm{Res}^{G_r^F}_{\bar{G}_r^F}(\sigma_1-\sigma_2)=1-\widetilde{\mathrm{St}}$ actually implies that $\mathrm{Res}^{G_r^F}_{\bar{G}_r^F}\sigma_1=1$ and $\mathrm{Res}^{G_r^F}_{\bar{G}_r^F}\sigma_2=\widetilde{\mathrm{St}}$, from which the assertion follows.
\end{proof}

For $r=1$, if $\theta$ is neither in general position nor trivial, one can use the commutativity between Deligne--Lusztig induction and $p$-constant class functions to determine $R_{T_1,U_1}^{\theta}$ (see the argument in \cite[Page~193]{DM_book_2nd_edition}). This commutativity has been generalised to $r\geq 2$ in \cite[Corollary~3.6]{Chen_2018_GreenFunction}, but it does not work in the above situation, as $\det$ is no more $p$-constant for $r\geq 2$.

\vspace{2mm} Proposition~\ref{prop:main} and Proposition~\ref{prop:GL_2} immediately imply:

\begin{coro}\label{coro:GL_2}
Assume that $\mathbb{G}=\mathrm{GL}_2$ and $T$ is Coxeter. Then
$$\left\{\dim R_{T_r,U_r}^{\theta} \mid  \theta\in\mathrm{Irr}(T_r^F) \right\}=\left\{(-1)^{i}(q-1)q^{i-1}\mid  i\in \{ 1,2,...,r \}  \right\};$$
moreover, if $|\dim R_{T_r,U_r}^{\theta}|>q-1$, then $R_{T_r,U_r}^{\theta}$ is irreducible up to sign, and the sign of $R_{T_r,U_r}^{\theta}$ is $(-1)^{1+\log_q\frac{|\dim R_{T_r,U_r}^{\theta}|}{q-1}}$.
\end{coro}

\begin{remark}
The dimensions and decompositions of  $R_{\bar{T}_r,\bar{U}_r}^{\bar{\theta}}$ (as a representation of  $\bar{G}_r^F=\mathrm{SL}_2(\mathcal{O}_r)$) are  the same as those described for $R_{T_r,U_r}^{\theta}$ in Proposition~\ref{prop:GL_2} (via Lemma~\ref{lemm:GL_n to SL_n}), except for the following cases depending on the parity of $q$:
\begin{itemize}
\item If $q$ is odd, the exception appears in (iii): When $\theta_0|_{\bar{T}_1^F}$ is the quadratic character, $-R_{\bar{T}_r,\bar{U}_r}^{\bar{\theta}}$ is a sum of two inequivalent irreducible representations of dimension $\frac{q-1}{2}$.

\item If $q$ is even, the exception appears in (i) and (ii); it suffices to describe the situation in (i): When $q$ is even, it can happen that $\theta$ is regular (or equivalently, $\bar{\theta}$ is regular in the sense of \cite{Lusztig2004RepsFinRings},\cite{Sta2009Unramified}) while $\bar{\theta}$ is \emph{not} in general position. Whenever this is the case, $-R_{\bar{T}_r,\bar{U}_r}^{\bar{\theta}}$ is the sum of two inequivalent irreducible constituents by \cite[Proposition~3.3]{Sta2009Unramified}; the constituents share the same dimension by Clifford's theorem, hence have dimension $\frac{q^r-q^{r-1}}{2}$. 
\end{itemize}  
Note that this agrees with the $\mathrm{SL}_2(\mathbb{F}_q[[\pi]]/\pi^2)$ case computed in \cite[Subsection~3.4]{Lusztig2004RepsFinRings}.
\end{remark}

\section{The sign of $R_{T_r,U_r}^{\theta}$}\label{sec:remark}

At this stage, it seems reasonable to expect: 

\begin{conj}
For general $(\mathbb{G},T,\theta,r)$, the sign of $R_{T_r,U_r}^{\theta}$  is
\begin{equation*}\label{conj:expectation}
\mathrm{sgn}(R_{T_r,U_r}^{\theta})=
(-1)^{ \left(\mathrm{rk}_q(T_1)+\mathrm{rk}_q(G_1)\right) \cdot  \left(1+\frac{\log_q|\dim R_{T_r,U_r}^{\theta}|_p}{\#\Phi^+}\right) },
\end{equation*}
where $p:=\mathrm{char}(\mathbb{F}_q)$, $(-)_p$ denotes the $p$-part of a positive integer, $\Phi^+$ denotes the set of positive roots of $G_1$ with respect to $T_1$, and $\mathrm{rk}_q(-)$ denotes the $\mathbb{F}_q$-rank of an algebraic group.
\end{conj}
 
Note that \eqref{conj:expectation} holds in the following cases: 
\begin{itemize}
\item $FU_r=U_r$;
\item $r=1$ (\cite{DL1976});
\item $q\geq 7$ (or $\mathbb{G}=\mathrm{GL}_n$ or $\mathrm{SL}_n$), with $\theta$ being strongly generic (\cite{ChenStasinski_2023_algebraisation_II});
\item $T$ is elliptic and $\theta=1$ (see Theorem~\ref{thm:main} and Remark~\ref{remark:Chan});
\item $\mathbb{G}=\mathrm{GL}_2$ or $\mathrm{SL}_2$ (Corollary~\ref{coro:GL_2} and Lemma~\ref{lemm:GL_n to SL_n}).
\end{itemize}
It would also be interesting to see if this conjecture extends to the parahoric setting in \cite{ChanetIvanov_CohomRepParahoricGp}.

{}

\vspace{10mm} \noindent {\bf Data availability statement.} We do not analyse or generate any datasets, because our work proceeds within a theoretical and mathematical approach.

\vspace{2mm} \noindent {\bf Conflict of interest statement.} We state that there is no conflict of interest.

\bibliographystyle{alpha}
\bibliography{zchenrefs}

\end{document}